\documentclass[]{amsart}
\usepackage{graphicx}
\usepackage{amssymb, amsmath, amsthm}

\newtheorem{theorem}{Theorem}[section]
\newtheorem{corollary}[theorem]{Corollary}
\newtheorem{lemma}[theorem]{Lemma}

\theoremstyle{definition}
\newtheorem{definition}[theorem]{Definition}
\newtheorem{example}[theorem]{Example}

\begin{document}
\title{Solid weak BCC-algebras}
\author{Wieslaw A. Dudek}
\address{Institute of Mathematics and Computer Science,
Wroc{\l}aw University of Techno\-logy, Wybrze\.ze Wyspia\'nskiego
27, 50-370 Wroc{\l}aw, Poland} \email{dudek@im.pwr.wroc.pl}
\thanks{{\scriptsize
\hskip -0.4 true cm MSC(2010): 03G25, 06F35
}}

\date{}

\begin{abstract}
We characterize weak BCC-algebras in which the identity
$(xy)z=(xz)y$ is satisfied only in the case when elements $x,y$
belong to the same branch.
\end{abstract}

\maketitle

\section{Introduction}

A goal that artificial intelligence has been trying to perfect for
decades is to realistically simulate decision making process of
humans when faced with both certain and uncertain types of
information. Logic behind those decisions is dominant in proof
theory. Such logic has two important roles both in mathematics and
computers -- it is a technique responsible for foundations of any
system as well as a tool for applications. Additionally to
classical logic there are various non-classical logic systems
built on top of it (e.g. many-valued logic, fuzzy logic, etc.)
that handle information with various aspects of uncertainty (cf.
\cite{Zad}) -- fuzziness, randomness and so on. Incomparability is
one important example of such uncertainty that may be easily
understood based on one's experience. Computer science relies
heavily on non-classical logic to deal with fuzzy and uncertain
information.

Because of advancements made in technology and artificial
intelligence, the study of $t$-norm-based logic systems and the
corresponding pseudo-logic systems has become a big focus in the
field of logic. BCK and BCI-algebras have been inspired by some
implicational logic and that can be event noticed when looking at
the similarity of names. Examples may be BCK-algebras and a BCK
positive logic or BCI-algebras and a BCI positive logic.
$t$-norm-bases algebraic investigations were indeed first to the
corresponding algebraic investigations and in the case of
pseudo-logic systems, algebraic development was first to the
corresponding logical development (see for example \cite{Ior}).
However, the link between such algebras and their respective
logics may be even much stronger. It is possible to define
translation procedures (and sometimes even their inverse
procedures) for all well formed formulas and all theorems of a
given logic into terms and theorems of the corresponding algebra
\cite{Ras}. Even when the full inverse translation procedure is
impossible, many researchers find it interesting and useful for
corresponding logics to study the algebras motivated by known
logics (cf. \cite{Bun}).

BCC-algebras (introduced by Y. Komori \cite{Ko'84}) are
generalizations of BCK-algebras, weak BCC-algebras are
generalizations of BCI-algebras. In view of strong connections
with a BIK$^+$-logic, (weak) BCC-algebras are also called (weak)
BIK$^+$-algebras (cf. \cite{Zh'07}). From results proved in
\cite{Du'90}, \cite{Du'92} and \cite{Mun} it follows that
MV-algebras are equivalent to the bounded commutative weak
BCC-algebras BCK-algebras, and so on. Hence, most of the algebras
related to the $t$-norm based logic, such as MTL-algebras,
BL-algebras, lattice implication algebras, Hilbert algebras and
Boolean algebras, are special cases of weak BCC-algebras. This
shows that BCC-algebras are considerably general structures.

Equivalence relations play an important role in the investigation
of such algebras and the corresponding logic. Equivalence
relations are determined by subsets called ideals. A special case
of ideals are branches induced by some elements. It is well known
that some properties which are satisfied in the ideals can be
extended to the whole algebra. Thus, it is sufficient to examine
these properties only on those ideals (branches).

The identity $(xy)z=(xz)y$ plays a crucial role in the study of
such algebras. It holds in all BCK-algebras and in some
generalizations of BCK-algebras, but not in BCC-algebras.
BCC-algebras satisfying this identity are BCK-algebras (cf.
\cite{Du'90} or \cite{Du'92}). The proofs of many properties of
BCK, BCI and BCC-algebra are required to meet the equality of all
elements of the algebra, but there are examples of algebras having
the property and not satisfying this identity. Therefore, it is
important to examine for which of the properties of weak
BCC-algebras it is sufficient that the equation $(xy)z=(xz)y$ is
satisfied by elements belonging to the same branch. This will
enable faster checking whether a given algebra has this property.

Below we begin the study of weak BCC-algebras in which the above
identity is satisfied only in the case when elements $x$ and $y$
belong to the same branch.

\section{Preliminaries}

\begin{definition} A {\it weak BCC-algebra} is a system $(X,*,0)$ of type $(2,0)$
satisfying the following axioms:
\begin{enumerate}
\item[$(i)$] \ $((x*y)*(z*y))*(x*z)= 0$,
\item[$(ii)$] \ $x*x = 0$,
\item[$(iii)$] \ $x*0 = x$,
\item[$(iv)$] \ $x*y = y*x = 0\Longrightarrow x = y$.
\end{enumerate}
\end{definition}

By many mathematicians, especially from China and Korea, weak
BCC-algebras are called {\it BZ-algebras} (cf. \cite{Zh'95},
\cite{ZWD} or \cite{DZW}), but we keep the first name because it
coincides with the general concept of names presented in the book
\cite{Ior} for algebras of logic.

A weak BCC-algebra satisfying the identity
\begin{enumerate}
\item[$(v)$] \ $0*x = 0$
\end{enumerate}
is called a {\it BCC-algebra} or {\it BIK$^+$-algebra}.

A weak BCC-algebra satisfying the identity
\begin{enumerate}
\item[$(vi)$] \ $(x*y)*z = (x*z)*y$
\end{enumerate}
is called a BCI-algebra. A weak BCC-algebra which is neither a
BCI-algebra or a BCC-algebra is called {\it proper}. Proper weak
BCC-algebras have at least four elements (see \cite{Du'96}). Note
that there are only two non-isomorphic weak BCC-algebras with four
elements:

\[
\begin{array}{c|cccc} *&0&1&2&3\\ \hline
0&0&0&2&2\\
1&1&0&2&2\\
2&2&2&0&0\\
3&3&3&1&0\\
\end{array}\\[3mm]
\hspace*{24mm}
\begin{array}{c|cccc} *&0&1&2&3\\ \hline
0&0&0&2&2\\
1&1&0&3&3\\
2&2&2&0&0\\
3&3&3&1&0\\
\end{array}
\]
\hspace*{32mm} Table\;1.  \hspace*{35mm} Table\;2.\\

They are proper, because in both cases $(3*2)*1 \neq (3*1)*2$.

A list of all non-isomorphic proper weak BCC-algebras having $5$
elements is presented in \cite{BZ-alg}. From results proved in
\cite{Du'96} and \cite{BZ-alg} we can deduce that for every
$n\geqslant 5$ one can find at least $22$ non-isomorphic proper
weak BCC-algebras having $n$ elements.

Any weak BCC-algebra can be considered as a partially ordered set
with the partial order $\leqslant$ defined by
\begin{equation}\label{1}
x\leqslant y \Longleftrightarrow x*y = 0.
\end{equation}

From $(i)$ it follows that in each weak BCC-algebra the
implications
\begin{equation} \label{2}
x \leqslant y \Longrightarrow x*z \leqslant y*z
\end{equation}
\begin{equation} \label{3}
x \leqslant y \Longrightarrow z*y \leqslant z*x
\end{equation}
are satisfied for all $x,y,z \in X$.

The set of all minimal (with respect to $\leqslant$) elements of
$X$ will be denoted by $I(X)$. This is characterized in \cite{DZW}
and \cite{BB}.

An important role in our investigations a will be played by two
subsets of $X$, namely:
\[
A(b)=\{x\in X\,|\,x\leqslant b\} \ \ \ \ \ {\rm and } \ \ \ \ \
B(a)=\{x\in X\,|\,a\leqslant x\},
\]
where $a\in I(X)$. The first subset is called the {\it initial
part} determined by $b\in X$, the second -- the {\it branch}
initiated by $a$.

A BCC-algebra has only one branch. Weak BCC-algebras defined by
Tables 1 and 2 have two minimal elements and two branches. Namely:
$I(X)=\{0,2\}$, $B(0)=\{0,1\}$, $B(2)=\{2,3\}$.

Branches initiated by different elements are disjoint (cf.
\cite{DZW}). A weak BCC-algebra is a set-theoretic union of
branches. Comparable elements are in the same branch, but there
are weak BCC-algebras containing branches in which not all
elements are comparable. Moreover, as it is proved in \cite{DZW},
elements $x$ and $y$ are in the same branch if and only if $x*y\in
B(0)$.

In theory of BCI-algebras an important role is played by
BCI-algebras satisfying some additional identities since such
BCI-algebras have properties similar to properties of lattices.
For example, BCI-algebras satisfying the identity
\begin{equation}\label{com}
x*(x*y)=y*(y*x),
\end{equation}
have properties similar to properties of lower semilattices (cf.
\cite{Huang}). A BCI-algebra in which
\begin{equation}\label{posimp}
(x*y)*y=x*y
\end{equation}
holds for all its elements has only one branch, so it is a
BCK-algebra. The class of BCK-algebras satisfying \eqref{posimp}
is a variety (cf. \cite{book}). Bounded BCK-algebras satisfying
the identity
\begin{equation}\label{imp}
x*(y*x)=x
\end{equation}
are distributive lattices (cf. \cite{book}). For BCI-algebras this
is not true. In connection with this fact, M. A. Chaudhry
initiated in \cite{Ch'92}, \cite{Ch'01} and \cite{Ch'02} the
investigation of BCI-algebras in which the above identities are
satisfied only by elements belonging to the same branch. The
obtained results are similar, but not identical, to results proved
earlier for BCI-algebras. Unfortunately, these results can not be
transferred to weak BCC-algebras because in the proofs of these
results the identity $(x*y)*z=(x*z)*y$ is used. But this identity
holds in a weak BCC-algebra only in the case when a weak
BCC-algebra is a BCI-algebra.

However, there are proper weak BCC-algebras in which the above
identities are satisfied by elements belonging to the same branch.

Keeping the terminology used in the theory of BCI/BCK-algebras we
will say that a weak BCC-algebra is {\it commutative} if it
satisfies \eqref{com}, {\it positive implicative} -- if it
satisfies \eqref{posimp}, and {\it implicative} -- if it satisfies
\eqref{imp}. If these equations are satisfied by elements
belonging to the same branch, then we will say that the
corresponding weak BCC-algebra is {\it branchwise commutative}
({\it branchwise positive implicative} or {\it branchwise
implicative}, respectively).

\begin{example}\label{Ex-1}\rm The following proper BCC-algebra
(calculated in \cite{Du'92})
\begin{center}$
\begin{array}{c|cccc} *&0&1&2&3\\ \hline\rule{0pt}{11pt}
0&0&0&0&0\\
1&1&0&0&1\\
2&2&2&0&1\\
3&3&3&3&0
\end{array}$\\[3mm]

Table\;3.
\end{center}

\noindent is an example of a BCC-algebra which is positive
implicative but it is neither commutative nor implicative.
\hfill$\Box{}$
\end{example}

\begin{example}\label{Ex-2}\rm The weak BCC-algebra defined by the
following table:

\begin{center}$
\begin{array}{c|cccccc} *&0&1&2&3&4&5\\ \hline\rule{0pt}{11pt}
0&0&0&4&4&2&2\\
1&1&0&4&4&2&2\\
2&2&2&0&0&4&4\\
3&3&2&1&0&4&4\\
4&4&4&2&2&0&0\\
5&5&4&3&3&1&0
\end{array}$\\[3mm]

Table\;4.
\end{center}

\noindent is proper because $(5*3)*2\ne (5*2)*3$. It has three
branches: $B(0)$, $B(2)$, $B(4)$. Direct computation shows that it
is branchwise commutative and branchwise implicative. Since
$1*(1*2)\ne 2*(2*1)$ and $1*(2*1)=4$ it is neither commutative nor
implicative. \hfill$\Box{}$
\end{example}

\section{Solid weak BCC-algebras}

A crucial role in the theory of BCI-algebras is played by the
equation
\begin{equation}\label{e-sol}
 (x*y)*z=(x*z)*y,
\end{equation}
which is valid for all $x,y,z\in X$. It is used in the proofs of
many basic facts. But in many proofs it is applied only to
elements belonging to some subsets. This motivates us to start
studying the BCC-algebras in which \eqref{e-sol} is fulfilled in
the case where at least two elements in this equation belong to
the same branch.

\begin{definition}
A weak BCC-algebra $X$ is called {\em left solid} (shortly: {\it
solid}) if the above equation is valid for all $x$, $y$ belonging
to the same branch and arbitrary $z\in X$. If it is valid for $y$,
$z$ belonging to the same branch and arbitrary $x\in X$, then we
say that a weak BCC-algebra $X$ is {\it right solid}. A left and
right solid weak BCC-algebra is called {\it supersolid}.
\end{definition}

Obviously, BCI-algebras and BCK-algebras are supersolid weak
BCC-alge\-bras. A solid BCC-algebra is a BCK-algebra since it has
only one branch. There are solid weak BCC-algebras which are not
BCI-algebras. For example, proper weak BCC-algebra defined by
Tables 1 and 2 are not solid, but the first is right solid, the
second is not right solid. A weak BCC-algebra defined by Table 4
is solid, but it is not right solid. It is the smallest solid weak
BCC-algebra because proper weak BCC-algebras with $5$ elements
(calculated in \cite{BZ-alg}) are not solid.

\begin{lemma}\label{L32}
In a solid weak BCC-algebra $X$ we have
\[
x*(x*y)\leqslant y
\]
for all $x,y$ from the same branch.
\end{lemma}
\begin{proof}
Indeed, according to the definition, for $x,y$ from the same
branch we have
\[
(x*(x*y))*y=(x*y)*(x*y)=0,
\]
which proves $x*(x*y)\leqslant y$.
\end{proof}
\begin{corollary}\label{C33}
In a solid weak BCC-algebra $x,y\in B(a)$ implies
$x*(x*y),\,y*(y*x)\in B(a)$.
\end{corollary}
\begin{proof}
Indeed, comparable elements are in the same branch (cf.
\cite{DZW}).
\end{proof}
\begin{corollary}\label{C34}
In a solid weak BCC-algebra $X$ for all $a\in I(X)$ and $x\in
B(a)$ we have \ $x*(x*a)=a$.
\end{corollary}
\begin{lemma}\label{L35}
In a solid weak BCC-algebra for elements belonging to the same
branch the following identity is satisfied:
\[
x*(x*(x*y))=x*y.
\]
\end{lemma}
\begin{proof}
Let $x,y\in B(a)$ for some $a\in I(X)$. Then $x*(x*y)\leqslant y$,
by Lemma \ref{L32}. This implies $x*(x*y)\in B(a)$ because
comparable elements belong to the same branch. Moreover, by
\eqref{3}, we also have $x*y\leqslant x*(x*(x*y))$.

Since $X$ is solid
$$
(x*(x*(x*y)))*(x*y)=(x*(x*y))*(x*(x*y))=0.
$$
Thus $x*(x*(x*y))\leqslant x*y$. This completes the proof.
\end{proof}

\begin{theorem} \label{T36} For a solid weak BCC-algebra $X$ the following conditions are equivalent:
\begin{enumerate}
\item[$(a)$] \ $X$ is branchwise commutative,
\item[$(b)$] \ $x*y=x*(y*(y*x))$ \ for $x,y$ from the same branch,
\item[$(c)$] \ $x=y*(y*x)$ \ for \ $x\leqslant y$,
\item[$(d)$] \ $x*(x*y)=y*(y*(x*(x*y)))$ \ for $x,y$ from the same branch.
\end{enumerate}
\end{theorem}
\begin{proof} $(a)\Longrightarrow (b)$ \ Let $x,y\in B(a)$ for some $a\in I(X)$. Then, by Corollary \ref{C33},
elements $y*(y*x)$ and $x*(x*y)$ are in
$B(a)$. Thus
$$
(x*(y*(y*x)))*(x*y)=(x*(x*y))*(y*(y*x))=0.
$$
Hence $x*(y*(y*x))\leqslant x*y$.

Since in view of \eqref{com}
\begin{equation}\label{7}
x*(x*(x*(x*y)))=(x*(x*y))*((x*(x*y))*x),
\end{equation}
we also have
\[
\arraycolsep=.5mm\begin{array}{rl}
(x*y)*(x*(y*(y*x)))&=(x*y)*(x*(x*(x*y)))\\[5pt]
&=(x*(x*(x*(x*y))))*y\\[3pt]
&\stackrel{(7)}{=}\{(x*(x*y))*((x*(x*y))*x)\}*y\\[5pt]
&=\{(x*(x*y))*((x*x)*(x*y))\}*y\\[5pt]
&=\{(x*(x*y))*(0*(x*y))\}*y\\[5pt]
&=\{(x*(x*y))*0\}*y\\[5pt]
&=(x*(x*y))*y=(x*y)*(x*y)=0,
\end{array}
\]
because $x*y\in B(0)$. Thus $x*y\leqslant x*(y*(y*x))$. This
together with the previous inequality gives $(b)$.

\medskip

$(b)\Longrightarrow (c)$ \  For $x\leqslant y$, from $(b)$, we
obtain $0=x*(y*(y*x))$. On the other hand, for all $x,y$ from the
same branch we have $(y*(y*x))*x=(y*x)*(y*x)=0$. Thus
$0=x*(y*(y*x))=(y*(y*x))*x$, which, by $(iv)$, implies $(c)$.

\medskip

$(c)\Longrightarrow (d)$ \  Let $x,y\in B(a)$ for some $a\in
I(G)$. Then $(x*(x*y))*y=0$, i.e., $x*(x*y)\leqslant y$. This, by
$(c)$, implies $(d)$.

\medskip

$(d)\Longrightarrow (a)$ \  If $x$ and $y$ are in the same branch
$B(a)$ and $(d)$ is satisfied, then we have
\begin{equation}\label{6}
(x*(x*y))*(y*(y*x))=(y*(y*(x*(x*y))))*(y*(y*x)).
\end{equation}
But $y*(y*x)\in B(a)$, by Lemma \ref{L32}, hence
\begin{equation}\label{7a}
(y*(y*(x*(x*y))))*(y*(y*x))=(y*(y*(y*x)))*(y*(x*(x*y))).
\end{equation}
Moreover, $y*(y*x)\in B(a)$ implies also
\[
0=(y*(y*x))*(y*(y*x))=(y*(y*(y*x)))*(y*x),
\]
which means that $y*(y*(y*x))\leqslant y*x$. Multiplying this
inequality by $y*(x*(x*y))$ and using \eqref{2} we obtain
\begin{equation}\label{8}
(y*(y*(y*x)))*(y*(x*(x*y)))\leqslant (y*x)*(y*(x*(x*y))).
\end{equation}
This together with \eqref{6}, \eqref{7a} and \eqref{8} proves that
\[
(x*(x*y))*(y*(y*x))\leqslant (y*x)*(y*(x*(x*y))).
\]
Since $x$ and $y$ are in the same branch, in view of $(d)$, we
obtain
\[
(y*x)*(y*(x*(x*y)))=(y*(y*(x*(x*y))))*x=(x*(x*y))*x=(x*x)*(x*y)=
0*(x*y).
\]
So,
\[
(x*(x*y))*(y*(y*x))\leqslant 0*(x*y)=0.
\]
Thus
\[
(x*(x*y))*(y*(y*x))=0.
\]
This shows that a solid weak BCC-algebra satisfying $(d)$ is
branchwise commutative.
\end{proof}

\begin{theorem} \label{T37} A solid weak BCC-algebra $X$ is branchwise commutative if and only
if
\begin{enumerate}
\item[$(a)$]  each branch of $X$ is a semilattice with respect to the operation
$\wedge$ defined by $x\wedge y=y*(y*x)$, or equivalently,
\item[$(b)$]  $A(x)\cap A(y)=A(x\wedge y)$ for all $x,y$ belonging to the same
branch.
\end{enumerate}
\end{theorem}
\begin{proof}
Let $X$ be a branchwise commutative solid weak BCC-algebra. Then,
by Lemma \ref{L32}, for $x,y\in B(a)$, $a\in I(X)$, we have
$y\wedge x=x\wedge y=x*(x*y)\leqslant y$. Hence $x\wedge y\in
B(a)$. Moreover, for any $p\leqslant q$ we have
$p=p*0=p*(p*q)=q\wedge p=p\wedge q$. Thus
\begin{equation}\label{10}
p\leqslant q\,\Longrightarrow\,p=p\wedge q .
\end{equation}

We prove that $p\wedge q$ is the greatest lower bound for any
$p,q\in B(a)$. For this consider an arbitrary element $z\in X$
such that $z\leqslant p$ and $z\leqslant q$. Then obviously $z\in
B(a)$ and
$$
z*(p\wedge q)=z*(q*(q*p))=(z \wedge q)*(q*(q*p))
$$
by \eqref{10}. Thus
$$
z*(p\wedge q)=(q*(q*z))*(q*(q*p))=(q*(q*(q*p)))*(q*z)
$$
because $q*(q*p)=p\wedge q\in B(a)$ for $p,q\in B(a)$. The last,
by Lemma \ref{L35}, is equal to $(q*p)*(q*z)$. Hence $z*(p\wedge
q)=(q*p)*(q*z)=0$ since $q*p\leqslant q*z$ for $z\leqslant p$. So,
$z\leqslant p\wedge q$. This means that $p\wedge q$ is the
greatest lower bound of $p$ and $q$. Hence the greatest lower
bound of each $p,q\in B(a)$ is in B(a), so $B(a)$ is a semilattice
with respect to $\wedge$. This proves $(a)$.

\medskip

Now, if $(a)$ is satisfied and $z\in A(x)\cap B(y)$ for some
$x,y\in B(a)$, $a\in I(X)$. Then $z\leqslant x$ and $z\leqslant
y$. Thus $z\in B(a)$ and $z\leqslant x\wedge y$, since $x\wedge
y\in B(a)$ is the greatest lower bound of $x$ and $y$. Hence $z\in
A(x\wedge y)$. Consequently, $A(x)\cap A(y)\subseteq A(x\wedge
y)$.

On the other hand, for any $z\in A(x\wedge y)$, by Lemma
\ref{L32}, we have $z\leqslant x\wedge y=y*(y*x)\leqslant x$, .
Hence $z\in A(x)$. Since
$$
(x\wedge y)*y=(y*(y*x))*y=(y*y)*(y*x)=0*(y*x)=0,
$$
we have $x\wedge y\leqslant y$ which implies $z\in A(y)$. So,
$z\in A(x)\cap A(y)$. Thus $A(x\wedge y)\subseteq A(x)\cap A(y)$
and hence $A(x\wedge y)=A(x)\cap A(y)$. This proves $(b)$.

Finally, if $(b)$ is satisfied, then
$$
A(x\wedge y)=A(x)\cap A(y)=A(y)\cap A(x)=A(y\wedge x).
$$
Hence $x\wedge y\in A(y\wedge x)$ and $y\wedge x\in A(x\wedge y)$.
Therefore $y\wedge x\leqslant x\wedge y$ and $x\wedge y\leqslant
y\wedge x$, which implies $x\wedge y=y\wedge x$. So, $X$ is
branchwise commutative.
\end{proof}

\begin{theorem}\label{T38}
A solid branchwise implicative weak BCC-algebra is branchwise
commutative.
\end{theorem}
\begin{proof}
Let $x,\,y\in B(a)$ for some $a\in I(X)$. Then $x*(x*y)\leqslant
y$, by Lemma \ref{L32}. Hence $x,\,y,\,x*(x*y)\in B(a)$. Thus
\begin{equation}\label{9}
(x*(x*y))*(y*(x*(x*y)))=x*(x*y)
\end{equation}
since $X$ is branchwise implicative. Moreover, from
$x*(x*y)\leqslant y$ and \eqref{2}, it follows
$$
(x*(x*y))*(y*(x*(x*y)))\leqslant y*(y*(x*(x*y))).
$$
Hence, by \eqref{9} and Lemma \ref{L32}, we obtain
$$
x*(x*y)\leqslant y*(y*(x*(x*y)))\leqslant x*(x*y).
$$
So, $y*(y*(x*(x*y)))=x*(x*y)$ for all $x,y\in B(a)$. Theorem
\ref{T36} completes the proof.
\end{proof}

\begin{theorem}\label{T39}
A solid branchwise implicative weak BCC-algebra is branchwise
positive implicative if and only if it is a commutative
BCK-algebra.
\end{theorem}
\begin{proof}
By Theorem \ref{T38}, a solid branchwise implicative weak
BCC-algebra is branchwise commutative. If it is branchwise
positive implicative, then for elements from the same branch we
have $x*y=(x*y)*y$, which for $x=y$ implies $0=0*y$. This means
that this weak BCC-algebra coincides with the branch $B(0)$.
Hence, it is a commutative BCK-algebra.

Conversely, if a commutative BCK-algebra is implicative, then
$$
((x*y)*y)*(x*y)=((x*y)*(x*y))*y=0*y=0
$$
and
$$
(x*y)*((x*y)*y)=y*(y*(x*y))=y*y=0,
$$
by commutativity and implicativity. Hence
$$
((x*y)*y)*(x*y)=(x*y)*((x*y)*y)=0,
$$
which implies $(x*y)*y=x*y$. So, $X$ is positive implicative.
\end{proof}

As a consequence of the above results we obtain well-known
characterization of implicative BCK-algebras presented below.

\begin{theorem}
A BCK-algebra is implicative if and only if it is both commutative
and positive implicative.
\end{theorem}
\begin{proof}
Indeed, in an implicative BCK-algebra $X$ we have
$$
x*y=(x*y)*(y*(x*y))=(x*y)*y
$$
for all $x,y\in X$. Hence, $X$ is positive implicative. By Theorem
\ref{T38} it also is commutative.

On the other hand, in any commutative and positive implicative
BCK-algebra $X$ we have
$$
x*(x*(y*x))=(y*x)*((y*x)*x)=(y*x)*(y*x)=0
$$
and
$$
(x*(y*x))*x=(x*x)*(y*x)=0*(y*x)=0
$$
for all $x,y\in X$. This implies $x*(y*x)=x$. So, $X$ is
implicative.
\end{proof}

\section{Positive implicative weak BCC-algebras}

Positive implicative BCK-algebras are defined as BCK-algebras
satisfying the identity \eqref{posimp}. By putting in
\eqref{posimp} $x=y$ we can see that in a weak BCC-algebra $X$ we
have $0*x=0$ for all $x\in X$. This means that $X=B(0)$. Thus a
weak BCC-algebra satisfying this identity has only one branch and
it is a BCC-algebra. If it is solid, then it is a BCK-algebra.
Hence (branchwise) positive implicative weak BCC-algebras should
be defined in another way. Positive implicative BCI-algebras are
defined as BCI-algebras satisfying the identity
$x*y=((x*y)*y)*(0*y)$. (Equivalent conditions one can find in the
book \cite{Huang}). Such defined positive implicative BCI-algebras
have properties similar to properties of positive implicative
BCK-algebras. Unfortunately, the proofs of these properties are
based on the identity \eqref{e-sol}. Only a small part of these
results can be transferred to solid and supersolid weak
BCC-algebras. In connection with this fact we introduce a new
class of weak BCC-algebras called {\it $\varphi$-implicative} weak
BCC-algebras.

For this we consider the self map $\varphi(x)=0*x$. This map was
formally introduced in \cite{DT} for BCH-algebras, but earlier was
used for \cite{Du'86} and \cite{Du'88} to investigations of some
classes of BCI-algebras. Later, in \cite{DZW}, it was used to
characterization of some ideals of weak BCC-algebras. Note that in
weak BCC-algebras $\varphi^2$ is an endomorphism ($\varphi$ is not
an endomorphism, in general). Moreover, $\varphi^2(x)\leqslant x$
for every $x\in X$ (see for example \cite{DZW}).

\begin{definition}
We say that a weak BCC-algebra $X$ is {\it $\varphi$-implicative}
if it satisfies the identity
$$
x*y=(x*y)*(y*\varphi^2(y)).
$$
If this identity is valid for elements belonging to the same
branch then such weak BCC-algebra is called {\it branchwise
$\varphi$-implicative}.
\end{definition}

Obviously, a BCK-algebra is positive implicative if and only if it
is $\varphi$-implicative. For BCI-algebras it is not true.

\begin{example} Consider the following BCI-algebra (cf.
\cite{Huang}):

\begin{center}$
\begin{array}{c|ccccc} *&0&1&2&3&4\\ \hline\rule{0pt}{11pt}
0&0&0&4&4&2\\
1&1&0&4&4&2\\
2&2&2&0&0&4\\
3&3&2&1&0&4\\
4&4&4&2&2&0
\end{array}$\\[3mm]

Table\;5.
\end{center}
This BCI-algebra has three branches: $B(0)$, $B(2)$ and $B(4)$.
Since $3*2\ne ((3*2)*2)*\varphi(2)$, it is not (branchwise)
positive implicative, but it is $\varphi$-implicative. Indeed, the
case $x\leqslant y$ is obvious because $\varphi^2(y)\leqslant y$
implies $y*\varphi^2(y)\in B(0)$. Similarly the case $y\in I(X)$.
The remaining five cases can be checked by standard simple
calculations. \hfill$\Box{}$
\end{example}

\begin{example}
Using the same method as in the previous example we can verify
that a solid weak BCC-algebra defined by Table 4 is branchwise
$\varphi$-implicative. It is not $\varphi$-implicative since
$5*3\ne (5*3)*(3*\varphi^2(3))$. \hfill$\Box{}$
\end{example}

One of the classical results is: {\it A BCK-algebra $($BCI-algebra
too$)$ is implicative if and only if it is commutative and
positive implicative} (cf. \cite{Huang}). For BCI-algebras this
result can be extended to the branchwise version but for weak
BCC-algebras this version is not valid. For weak BCC-algebras it
will be valid in the branchwise $\varphi$-implicative version. For
details see \cite{DT2}.

\section{Conclusions}
Solid BCC-algebras are difficult to study due to the fact that
equality $(x*y)*z=(x*z)*y$ may not be true for elements belonging
to different branches. However, using the calculation method
presented above we can obtain results very similar to those
obtained for BCI-algebras. In some cases, like for example in the
first part of proof of the Theorem \ref{T36}, our method comes
down to almost mechanical transformations. In other cases we have
to continue doing the transformations and choose elements so that
the result is still in the same branch. This is troublesome. In
exchange the results are true for the wider class of algebras.

Further investigations of solid weak BCC-algebras are continued in
\cite{DT2}, \cite{Thom} and \cite{ThZh}. In the first paper a
description of the $\varphi$-implicative algebras is given. In the
second paper results related to the $f$-derivations of weak
BCC-algebras are proved. It turns out that they are very similar
to those proved in \cite{ZL} for BCI-algebras. Results presented
in the third paper show that verification of various properties of
ideals can be reduced to verification of these properties in some
branches which is very important for computer verification.

\end{document}